\documentclass[a4paper,reqno,10pt]{amsart}
\usepackage[a4paper,left=2cm,right=2cm,top=2cm]{geometry}
\usepackage[onehalfspacing]{setspace}
\usepackage[colorlinks=true, linkcolor=magenta, citecolor=cyan]{hyperref}
\usepackage{fontspec}
\usepackage{amsmath}
\usepackage{amsthm}
\usepackage{amssymb}
\usepackage{amsfonts}
\usepackage{mathtools}
\usepackage{colonequals}
\usepackage{dsfont}
\usepackage{enumitem}
\usepackage{mathdots}
\usepackage{fancyhdr}
\usepackage{mdframed}
\usepackage{csquotes}
\usepackage{bm}
\usepackage[normalem]{ulem}

\usepackage{tikz-cd}
\tikzset{
  symbol/.style={
    draw=none,
    every to/.append style={
      edge node={node [sloped, allow upside down, auto=false]{$#1$}}}
  }
}

\uccode`ß="1E9E

\newtheorem{introtheorem}{Theorem}

\newtheorem{introproposition}[introtheorem]{Proposition}

\newtheorem{theorem}{Theorem}[section]
\newtheorem{proposition}[theorem]{Proposition}
\newtheorem{lemma}[theorem]{Lemma}
\newtheorem{corollary}[theorem]{Corollary}

\theoremstyle{definition}

\theoremstyle{remark}
\newtheorem{remark}[theorem]{Remark}
\newtheorem{example}[theorem]{Example}
\newtheorem{addendum}[theorem]{Addendum}

\usepackage[backend=biber,maxalphanames=5,maxbibnames=4,style=alphabetic]{biblatex}
\AtBeginBibliography{\small}

\newcommand{\br}[1]{\left(#1\right)}
\newcommand{\sqb}[1]{\left\lbrack #1\right\rbrack}
\newcommand{\abs}[1]{\left\lvert #1\right\rvert}
\newcommand{\iso}{\stackrel{\sim}{\rightarrow}}

\newcommand{\N}{\mathbb{N}}

\newcommand{\C}{\mathcal{C}}
\newcommand{\D}{\mathcal{D}}
\newcommand{\E}{\mathcal{E}}

\newcommand{\I}{\mathcal{I}}
\newcommand{\K}{\mathcal{K}}

\newcommand{\T}{\mathcal{T}}

\newcommand{\X}{\mathcal{X}}
\newcommand{\Y}{\mathcal{Y}}

\newcommand{\const}{\mathrm{const}}

\newcommand{\Cocart}{\mathrm{Cocart}}
\newcommand{\LFib}{\mathrm{LFib}}

\DeclareMathOperator*{\colim}{colim}

\newcommand{\Map}{\operatorname{Map}}

\newcommand{\id}{\operatorname{id}}
\newcommand{\ev}{\operatorname{ev}}
\newcommand{\Fun}{\operatorname{Fun}}

\newcommand{\mhyphen}{\operatorname{-}}
\newcommand{\Un}{\mathrm{Un}}

\newcommand{\Cat}{\mathrm{Cat}}
\newcommand{\An}{\mathrm{An}}
\newcommand{\Sp}{\mathrm{Sp}}

\renewcommand{\Pr}{\mathrm{Pr}}
\newcommand{\Acc}{\mathrm{Acc}}

\newcommand{\Pro}{\operatorname{Pro}}

\newcommand{\RTop}{\mathrm{RTop}}
\newcommand{\Ar}{\mathrm{Ar}}

\newcommand{\lax}{\mathrm{lax}}
\newcommand{\oplax}{\mathrm{oplax}}
\newcommand{\op}{\mathrm{op}}
\newcommand{\cc}{\mathrm{cc}}

\newcommand{\ca}{\mathrm{ca}}
\newcommand{\st}{\mathrm{st}}

\newcommand{\singlequote}[1]{\textquotesingle #1\textquotesingle}
\newcommand{\Adj}{\arrow[r,shift left=0.2em]\arrow[r,leftarrow,shift right=0.2em]}

\title{On Cofiltered Limits of $\infty$-Categories and Adjunctions}
\author{Thorger Gei\ss}
\address{FB Mathematik und Informatik, Universität Münster, Einsteinstraße 62, 48149 Münster, Germany}
\email{tgeiss@uni-muenster.de}
\begin{document}

\begin{abstract}
We prove that, under mild assumptions, the limit of a cofiltered diagram of $\infty$-categories is a reflective (resp. coreflective) localization of its oplax (resp. lax) limit. This gives rise to a number of exceptional \singlequote{stability} results for categorical properties under such limits, in particular one for adjunctions. As an application, we recover a push-pull formula describing certain filtered colimits in the $\infty$-category $\Pr^L$ of presentable $\infty$-categories.
\end{abstract}

\setcounter{tocdepth}{1}

\begingroup\parskip=0pt
\maketitle
\tableofcontents
\endgroup

\section{Introduction}

In the study of $\infty$-category theory (\cite{Joyal}, \cite{HTT}, \cite{HA}), one of the most important tools to construct or describe $\infty$-categories is to take limits in the (huge) $\infty$-category $\widehat{\Cat}_{\infty}$ of (large) $\infty$-categories (e.g. stabilization, various flavors of completions, descent properties, etc.). It is therefore natural to ask which categorical properties are compatible with taking limits. In the present paper, we address this question for an exceptional class of limits—\textit{cofiltered} limits—and, among other things, one of the central relationships between two $\infty$-categories, namely an \textit{adjunction}.

To set the stage, let $\I$ be a small $\infty$-category, $\C_{\bullet},\D_{\bullet}\colon\I^{\op}\rightarrow\widehat{\Cat}_{\infty}$ be two diagrams of $\infty$-categories and $R\colon\D_{\bullet}\Rightarrow\C_{\bullet}$ a natural transformation that is pointwise given by right adjoints $R_i\colon\D_i\rightarrow\C_i,\,i\in\I$. The pointwise left adjoints $L_i\colon\C_i\rightarrow\D_i,\,i\in\I$ generally do not assemble into a natural transformation $L\colon\C_{\bullet}\rightarrow\D_{\bullet}$ as they need not be compatible with the transition maps. This compatibility is precisely the \textit{Beck-Chevalley condition} for the naturality squares and, if satisfied, indeed implies that $L$ is a natural transformation and $\begin{tikzcd}\lim_{\I^{\op}}L\colon\lim_{\I^{\op}}\D_{\bullet}\Adj & \lim_{\I^{\op}}\C_{\bullet}\colon\lim_{\I^{\op}}R\end{tikzcd}$ is an adjunction (\cite[Proposition 5.1]{HorevYanovski}, \cite[Lemma 3.19]{Linskens}). In general, this is the best one can hope for.

\subsection*{Main Result} The main result of this paper is that, exceptionally, in the case that $\I$ is a filtered $\infty$-category, no assumption on the natural transformation $R$ is necessary to obtain such a positive result: the failure of the Beck-Chevalley condition can, under a mild technical assumption on the diagram $\D_{\bullet}$, be amended by taking suitable filtered colimits of the pointwise left adjoints along the mate transformations of the naturality squares.

\begin{introtheorem}[Theorem \ref{AdjunctionDual}]\label{IntroAdjunction}
Let $\I$ be a small, filtered $\infty$-category and, in the above setting, suppose that the category $\D_i$ admits $\I$-shaped colimits for all $i\in\I$ and the transition maps $\varphi^{\ast}\colon\D_j\rightarrow\D_i$ preserve $\I$-shaped colimits for all $\varphi\colon i\rightarrow j$ in $\I$. Then there is an induced adjunction
\begin{equation*}
\begin{tikzcd}
L\colon\lim_{\I^{\op}}\C_{\bullet}\Adj & \lim_{\I^{\op}}\D_{\bullet}\colon\lim_{\I^{\op}}R.
\end{tikzcd}
\end{equation*}
More precisely, the left adjoint $L$ is explicitly given levelwise by
\begin{equation*}
\pi_iL\simeq\colim_{(\varphi\colon i\rightarrow j)\in\I_{i/}}\varphi^{\ast}L_j\pi_j\colon\lim\nolimits_{\I^{\op}}\C\rightarrow\D_i\qquad\text{ for all }i\in\I.
\end{equation*}
\end{introtheorem}

\subsection*{Applications} In Section \ref{Applications}, we illustrate a number of applications of Theorem \ref{IntroAdjunction}. In particular, we obtain \singlequote{stability} results for categorical properties which can be expressed in terms of adjunctions under cofiltered limits, e.g. for categories that are cocomplete (Corollary \ref{ApplicationCocomplete}) or cartesian closed (Proposition \ref{CartesianClosed}). Using the explicit formula, we also obtain such results for the left adjoint $L$ being fully faithful (Proposition \ref{FullyFaithful}) or compatible with finite limits (Proposition \ref{FiniteLimits}).

We highlight two further applications. The first concerns filtered colimits in the $\infty$-category $\Pr^L$ of presentable $\infty$-categories and left adjoint functors. To this end, let $\C_{\bullet}\colon\I\rightarrow\Pr^L$ be a small, filtered diagram and recall that $\colim_{\I}^{\Pr^L}\C_{\bullet}$ is computed as the limit $\lim_{\I^{\op}}\C_{\bullet}^R$ of the diagram $\C_{\bullet}^R\colon\I^{\op}\rightarrow\widehat{\Cat}_{\infty}$ obtained by passing to right adjoints. The structure map $\iota_i\colon\C_i\rightarrow\colim_{\I}^{\Pr^L}\C_{\bullet}\simeq\lim_{\I^{\op}}\C_{\bullet}^R$ may then be understood levelwise via the projections $\pi_j\colon\lim_{\I^{\op}}\C_{\bullet}^R\rightarrow\C_j$, and we interpret $\pi_j\iota_i$ as pushing from stage $i$ to the colimit and then pulling back to stage $j$. Under the hypotheses of Theorem \ref{IntroAdjunction}, this push-pull can be expressed as a filtered colimit of push-pulls through the finite stages of the diagram.

\begin{introproposition}[Push-Pull Formula, Proposition \ref{PushPullFormula} \& Remark \ref{FilteredColimitPrL}]\label{IntroPushPull}
Let $\I$ be a small, filtered $\infty$-category and $\C_{\bullet}\colon\I\rightarrow\Pr^L$ a diagram such that the right adjoint $\varphi^{\ast}\colon\C_j\rightarrow\C_i$ to the transition map $\varphi_{\#}$ preserves $\I$-shaped colimits for all $\varphi\colon i\rightarrow j$ in $\I$. For any $i\in\I$ and $j\in\I$, the canonical comparison map
\begin{equation*}
\begin{tikzcd}
\colim\limits_{(\varphi\colon i\rightarrow k,\psi\colon j\rightarrow k)\in\I_{\{i,j\}/}}\psi^{\ast}\varphi_{\#}\arrow[r,"\sim"] & \pi_j\iota_i
\end{tikzcd}
\end{equation*}
induced by the adjunction units is a natural equivalence of functors $\C_i\rightarrow\C_j$.
\end{introproposition}

The second application concerns the shape theory of $\infty$-topoi and is a straightforward consequence of the push-pull formula. The statement is adapted from Hoyois \cite[Proposition 2.11, (1)]{Hoyois}.

\begin{introproposition}[Proposition \ref{ShapeCofiltered}]\label{IntroShapeCofiltered}
The shape functor $\Pi_{\infty}\colon\RTop\rightarrow\Pro(\An)$ preserves small, cofiltered limits of diagrams of compact $\infty$-topoi and perfect geometric morphisms.
\end{introproposition}

\subsection*{Proof Strategy} To prove Theorem \ref{IntroAdjunction}, we return to the failure of the Beck-Chevalley condition and note that the left adjoints do assemble into an \textit{oplax} natural transformation $L\colon\C_{\bullet}\Rightarrow^{\oplax}\D_{\bullet}$, which then induces an adjunction
\begin{equation*}
\begin{tikzcd}
\lim_{\I^{\op}}^{\oplax}L\colon\lim_{\I^{\op}}^{\oplax}\C_{\bullet}\Adj & \lim_{\I^{\op}}^{\oplax}\D_{\bullet}\colon\lim_{\I^{\op}}^{\oplax}R.
\end{tikzcd}
\end{equation*}
The technical crux of our proof is now that, in the case of a cofiltered diagram satisfying the hypotheses of Theorem \ref{IntroAdjunction}, the inclusion of the ordinary limit into the oplax limit is a reflective localization.

\begin{introtheorem}[Theorem \ref{MainThmDual}]\label{IntroReflection}
In the situation of Theorem \ref{IntroAdjunction}, the full subcategory inclusion $\lim_{\I^{\op}}\D_{\bullet}\subseteq\lim_{\I^{\op}}^{\oplax}\D_{\bullet}$ admits a left adjoint $L\colon\lim^{\oplax}_{\I^{\op}}\D_{\bullet}\rightarrow\lim_{\I^{\op}}\D_{\bullet}$, which is explicitly given levelwise by
\begin{equation*}
\pi_iL\simeq\colim_{(\varphi\colon i\rightarrow j)\in\I_{i/}}\varphi^{\ast}\pi_j\colon\lim\nolimits_{\I^{\op}}^{\oplax}\D_{\bullet}\rightarrow\D_i\qquad\text{ for all }i\in\I.
\end{equation*}
Furthermore, the resulting adjunction is natural in the diagram $\D_{\bullet}$.
\end{introtheorem}

%Again, the failure of a section to not be cocartesian can be amended by replacing it with a suitable filtered colimit along its cocartesian transports.

\subsection*{Relation to other work}

The case $\I=\N$ of Theorem \ref{IntroAdjunction} appeared previously in the author's (unpublished) Master's thesis \cite[Appendix B]{Geiss}. The present paper extends this to all small, filtered $\infty$-categories and conceptualizes it.

Though our main results have been obtained independently, closely related statements appear elsewhere in the literature. A version of the construction in Theorem \ref{IntroReflection} appears in an unpublished note of Gaitsgory \cite[Lemma 1.3.6]{Gaitsgory} and a version of  the push-pull formula (Proposition \ref{IntroPushPull}) appears in recent work of Efimov \cite[Proposition 1.66,2)]{Efimov}, though either in a slightly different setting. Furthermore, Proposition \ref{IntroShapeCofiltered} is adapted from Hoyois \cite[Proposition 2.11, (1)]{Hoyois}. These references, however, only assert or sketch the relevant statements and do not provide a detailed construction or verification of the relevant higher coherences. To provide some of these details and explicate the underlying categorical mechanism is one of the purposes of the present paper.

A generalization of the push-pull formula (Proposition \ref{IntroPushPull}) to indexing $\infty$-categories $\I$ that are not filtered will be the subject of future joint work.

\subsection*{Conventions}

We will make use of the following conventions:

\begin{itemize}
\item The prefix \singlequote{$\infty$-} will be dropped henceforth, i.e. \singlequote{category} always refers to $\infty$-category.
\item The ($\infty$-)category of spaces/($\infty$-)groupoids will be referred to as the category of \textit{anima}, denoted $\An$.
\item The (large resp. huge) ($\infty$-)category of (small resp. large) ($\infty$-)categories is denoted $\Cat$ resp. $\widehat{\Cat}$.
\item The ($\infty$-)category of large ($\infty$-)categories and left resp. right adjoint functors is denoted $\widehat{\Cat}^L$ resp. $\widehat{\Cat}^R$. Its full subcategory on the presentable ($\infty$-)categories is denoted $\Pr^L$ resp. $\Pr^R$.
\item If $\C$ is an ($\infty$-)category and $f\colon x\rightarrow y$ a morphism in $\C$, then post- resp. pre-composition with $f$ will be denoted $f_{\circ}\colon\Map_{\C}(t,x)\rightarrow\Map_{\C}(t,y)$ resp. $f^{\circ}\colon\Map_{\C}(y,t)\rightarrow\Map_{\C}(x,t)$ for any $t\in\C$.
\item If $p\colon\E\rightarrow\C$ is a cocartesian resp. cartesian fibration of ($\infty$-)categories and $f\colon x\rightarrow y$ a morphism in $\C$, then we denote by $f_{\#}\colon\E_x\rightarrow\E_y$ resp. $f^{\ast}\colon\E_y\rightarrow\E_x$ the cocartesian resp. cartesian transport.
\item For any collection $\K$ of small ($\infty$-)categories, $\widehat{\Cat}(\K)$ resp. $\widehat{\Cat}\sqb{\K}$ is the ($\infty$-)category of ($\infty$-)categories admitting $\K$-shaped colimits resp. limits and functors preserving these. The (non-full) subcategory inclusions $\widehat{\Cat}(\K)\hookrightarrow\widehat{\Cat}$ resp. $\widehat{\Cat}\sqb{\K}\hookrightarrow\widehat{\Cat}$ create (small) limits.
\end{itemize}

\subsection*{Acknowledgments}

I would like to thank Phil Pützstück for helpful conversations, Georg Lehner for suggesting the topos-theoretic applications as well as pointing me to the relevant literature and Tobias Lenz for Example \ref{TobiasExample}. Furthermore, I would like to thank my advisor Thomas Nikolaus for his advice and continued support.

The author was funded by the Deutsche Forschungsgemeinschaft (DFG, German Research Foundation) – Project-ID 427320536 – SFB 1442, as well as under Germany’s Excellence Strategy EXC2044/2–390685587, Mathematics Münster: Dynamics–Geometry–Structure.

\section{Preliminaries}

In this section, we discuss a key technical input to our proof of Theorem \ref{IntroReflection}. Recall that a functor $\varphi\colon\C\rightarrow\D$ is called \textit{cofinal} resp. \textit{coinitial} if pre-composition with $\varphi$ does not change colimits resp. limits.\footnote{See \cite[Table 6.1]{HaugsengYetAnother} for varying terminological conventions.} This is equivalent, by a version of Quillen's Theorem A, to the relative slices $\C\times_{\D}\D_{d/}$ resp. $\C\times_{\D}\D_{/d}$ being weakly contractible for all $d\in\D$. Furthermore, a functor $p\colon\E\rightarrow\I$ is called \textit{smooth} resp. \textit{proper} if the pullback functor $p^{\ast}\colon\Cat_{/\I}\rightarrow\Cat_{/\E}$ preserves morphisms that lie over cofinal resp. coinitial morphisms in $\Cat$.

The next lemma follows from \cite[Lemma 5.5]{HHNLSpans} and \cite[Corollary 7.7.10]{HaugsengYetAnother}, but we give a direct proof.

\begin{lemma}\label{RealizationFibWeakEquivaence}
Let $p\colon\E\rightarrow\I$ be a cocartesian fibration. Then the following are equivalent:
\begin{enumerate}[label=(\alph*)]
\item The fibers of $p$ are weakly contractible.
\item $p$ is a universal weak equivalence, i.e. every base-change of $p$ is a weak equivalence.
%\begin{tikzcd}
%\E^{\prime}\arrow[r]\arrow[d,"p^{\prime}"]\pb & \E\arrow[d,"p"]\\
%\I^{\prime}\arrow[r] & \I.
%\end{tikzcd}
%\end{equation*}
\end{enumerate}
\end{lemma}
\begin{proof}[Proof.]
It is clear that (b) implies (a) by considering the cospan $\E\rightarrow\I\leftarrow\{i\}$ for any $i\in\I$. In the converse direction, condition (a) is clearly stable under base-change, so it suffices to prove that $p$ itself is a weak equivalence. The left adjoint $(-)^l\colon\Cocart(\I)\rightarrow\LFib(\I)$ maps $p\colon\E\rightarrow\I$ to an equivalence $p^l\colon\E^l\rightarrow\I$, i.e. the terminal left fibration. This implies the conclusion as the unit map $\eta\colon\E\rightarrow\E^l$ over $\I$ is always a weak equivalence, e.g. since straightening-unstraightening identifies its geometric realization $|\eta|\colon|\E|\rightarrow|\E^l|$ with the canonical equivalence $|\colim_{i\in\I}\E_i|\rightarrow\colim_{i\in\I}|\E_i|$.
\end{proof}

\begin{proposition}\label{CocartesianProper}
Let $p\colon\E\rightarrow\I$ be a cocartesian fibration. Then:
\begin{enumerate}[label=(\roman*)]
\item $p$ is proper if and only if the cocartesian transport $f_{\#}\colon\E_i\rightarrow\E_j$ is coinitial for any $f\colon i\rightarrow j$ in $\I$.
\item $p$ is smooth.
\end{enumerate}
\end{proposition}
\begin{remark}
The case (ii) is well-known, see \cite[Proposition 4.1.2.15]{HTT}, and we've included it only for completeness. The case (i) is, to our knowledge, new and the one that we actually need in the next section.
\end{remark}
\begin{proof}[Proof.]
To prove (i), we first show necessity, so assume that $p$ is proper. Let $f\colon i\rightarrow j$ be any morphism in $\I$ and consider the diagram
\begin{equation*}
\begin{tikzcd}
\E_i\arrow[r,hookrightarrow]\arrow[d]\arrow[dr,very near start,phantom,"\lrcorner"] & \E_f\arrow[r]\arrow[d]\arrow[dr,very near start,phantom,"\lrcorner"] & \E\arrow[d,"p"]\\
\{0\}\arrow[r,hookrightarrow] & \sqb{1}\arrow[r,"f"] & \I,
\end{tikzcd}
\end{equation*}
in which the squares are cartesian and $\{0\}\hookrightarrow\sqb{1}$ is coinitial. Then $\E_i\hookrightarrow\E_f$ is coinitial by assumption, but $f_{\#}\colon\E_i\rightarrow\E_j$ factors as the composite of $\E_i\hookrightarrow\E_f$ and the left adjoint $\E_f\rightarrow\E_j$ to the fiber inclusion. Then, since left adjoints are coinitial, it follows that the composite $f_{\#}$ is also coinitial.

Next, we show sufficiency, so assume that $f_{\#}\colon\E_i\rightarrow\E_j$ is coinitial for any $f\colon i\rightarrow j$ in $\I$. To check that $p$ is proper, consider any diagram
\begin{equation*}
\begin{tikzcd}
\C^{\prime}\arrow[r,"\varphi^{\prime}"]\arrow[d]\arrow[dr,very near start,phantom,"\lrcorner"] & \D^{\prime}\arrow[r]\arrow[d]\arrow[dr,very near start,phantom,"\lrcorner"] & \E\arrow[d,"p"]\\
\C\arrow[r,"\varphi"] & \D\arrow[r] & \I,
\end{tikzcd}
\end{equation*}
in which the squares are cartesian and $\varphi$ is coinitial. For any $d^{\prime}\in\D^{\prime}$ with images $d\in\D$ and $e\in\E$, we have
\begin{equation*}
\C^{\prime}\times_{\D^{\prime}}\D^{\prime}_{/d^{\prime}}\simeq\C\times_{\D}\D^{\prime}_{/d^{\prime}}\simeq(\C\times_{\D}\D_{/d})\times_{\I_{/p(e)}}\E_{/e}.
\end{equation*}
Since $\varphi$ is coinitial, it suffices to show that $\C^{\prime}\times_{\D^{\prime}}\D^{\prime}_{/d^{\prime}}\rightarrow\C\times_{\D}\D_{/d}$ is a weak equivalence to conclude that $\varphi^{\prime}$ is also coinitial. This follows by Lemma \ref{RealizationFibWeakEquivaence} as $p^{\prime}\colon\E_{/e}\rightarrow\I_{/p(e)}$ is a cocartesian fibration by \cite[\href{https://kerodon.net/tag/0477}{Tag 0477}]{kerodon} and its fibers are indeed weakly contractible: the fiber of $p^{\prime}$ over $f\colon i\rightarrow p(e)$ is given by the pullback $\E_i\stackrel{f_{\#}}{\rightarrow}\E_{p(e)}\leftarrow(\E_{p(e)})_{/e}$, which is weakly contractible precisely since $f_{\#}$ is coinitial.

To prove (ii), we repeat the same argument using a cofinal map $\varphi$ instead. There, we have to apply Lemma \ref{RealizationFibWeakEquivaence} to $p^{\prime}\colon\E_{e/}\rightarrow\I_{p(e)/}$ instead, which is also a cocartesian fibration by \cite[\href{https://kerodon.net/tag/0477}{Tag 0477}]{kerodon} and whose fiber over $f\colon p(e)\rightarrow i$ is given by the slice $(\E_i)_{f_{\#}(e)/}$, which is always weakly contractible as it has an initial object.
\end{proof}

This proposition arose out of verifying the following example, which the author learned from Tobias Lenz.
\begin{example}\label{TobiasExample}
The map of posets $\min\colon[1]\times[1]\rightarrow[1]$ is the cocartesian fibration classifying the functor $[1]\rightarrow\Cat$ that selects the arrow $\Lambda_0^2\rightarrow\{\ast\}$. This is coinitial, since $\Lambda_0^2$ is weakly contractible, but does not admit a right adjoint, since $\Lambda_0^2$ does not have a terminal object. Thus, $\min$ is proper, but not a cartesian fibration.
\end{example}

\section{The Main Results}

In this section, we prove our main results. To fix notation, let $\Un\colon\Fun(\I,\widehat{\Cat})\iso\widehat{\Cocart}(\I)$ denote the (cocartesian) unstraightening equivalence. Recall that the limit of a diagram $\C_{\bullet}\colon\I\rightarrow\widehat{\Cat}$ is computed as the category $\Fun_{/\I}^{\cc}(\I,\Un(\C_{\bullet}))$ of cocartesian sections of the cocartesian unstraightening $p\colon\Un(\C_{\bullet})\rightarrow\I$. This is a full subcategory of the category $\Fun_{/\I}(\I,\Un(\C_{\bullet}))$ of \textit{all} sections of the cocartesian unstraightening $p\colon\Un(\C_{\bullet})\rightarrow\I$, which instead computes the \textit{lax} limit of $\C_{\bullet}$, denoted $\lim^{\lax}_{\I}\C_{\bullet}$.\footnote{This definition is all that is relevant to us, but see \cite{GHN} or \cite{HHNL} for more on lax limits.} For any $i\in\I$, we will denote by $\pi_i\colon\lim^{\lax}_{\I}\C_{\bullet}\rightarrow\C_i$ resp. $\pi_i\colon\lim_{\I}\C_{\bullet}\rightarrow\C_i$ the canonical projection to the $i$-th stage of the diagram, given by the evaluation of sections at $i$. To minimize the number of $\op$'s appearing, we will instead prove the dual versions of Theorems \ref{IntroAdjunction} and \ref{IntroReflection} and then dualize again.

\begin{theorem}\label{MainThm}
Let $\I$ be a small, cofiltered category and $\C_{\bullet}\colon\I\rightarrow\widehat{\Cat}\sqb{\I}$ a diagram. Then $\lim_{\I}\C_{\bullet}\subseteq\lim_{\I}^{\lax}\C_{\bullet}$ admits a right adjoint $R\colon\lim^{\lax}_{\I}\C_{\bullet}\rightarrow\lim_{\I}\C_{\bullet}$, which is given levelwise by
\begin{equation*}
\pi_iR\simeq\lim_{(\varphi\colon j\rightarrow i)\in\I_{/i}}\varphi_{\#}\pi_j\colon\lim\nolimits_{\I}^{\lax}\C_{\bullet}\rightarrow\C_i\qquad\text{ for all }i\in\I.
\end{equation*}
Furthermore, the resulting adjunction is natural in the diagram $\C_{\bullet}$.
\end{theorem}
\begin{proof}[Proof.]
We construct the adjoint $R$ locally, i.e. for every section $t\colon\I\rightarrow\Un(\C_{\bullet})$ of the cocartesian fibration $p\colon\Un(\C_{\bullet})\rightarrow\I$ we construct a cocartesian section $R(t)\colon\I\rightarrow\Un(\C_{\bullet})$ of $p$ and a natural equivalence
\begin{equation*}
\Map_{\Fun_{/\I}^{\cc}(\I,\Un(\C_{\bullet}))}(-,R(t))\simeq\Map_{\Fun_{/\I}(\I,\Un(\C_{\bullet}))}(-,t)\colon\Fun_{/\I}^{\cc}(\I,\Un(\C_{\bullet}))^{\op}\rightarrow\An.
\end{equation*}
The Yoneda lemma then (uniquely) assembles these $R(t)$ into a functor $R\colon\Fun_{/\I}(\I,\Un(\C_{\bullet})\rightarrow\Fun_{/\I}^{\cc}(\I,\Un(\C_{\bullet})$ in such a manner that the above equivalence becomes natural in $t$, presenting the desired adjunction.

First, we consider the cocartesian lifting problem
\begin{equation*}
\begin{tikzcd}
\Ar(\I)\times\{0\}\arrow[r,"t\ev_0"]\arrow[d] & \Un(\C_{\bullet})\arrow[d,"p"]\\
\Ar(\I)\times\sqb{1}\arrow[r,"\ev"]\arrow[ur,"\tilde{t}",dotted] & \I,
\end{tikzcd}
\end{equation*}
which admits the indicated solution. Using this map, we construct the lifting-extension problem
\begin{equation*}
\begin{tikzcd}
\Ar(\I)\arrow[r,"\tilde{t}_1"]\arrow[d] & \Un(\C_{\bullet})\arrow[d,"p"]\\
\Ar(\I^{\triangleleft})\times_{\I^{\triangleleft}}\I\arrow[r,"\ev_1"]\arrow[ur,dotted,"\overline{t}"] & \I,
\end{tikzcd}
\end{equation*}
which we will solve by a $p$-right Kan extension. To see that this exists, we use (the dual of) \cite[\href{https://kerodon.net/tag/0302}{Tag 0302}]{kerodon}. For any object $(-\infty\rightarrow i)\in\Ar(\I^{\triangleleft})\times_{\I^{\triangleleft}}\I$ (the condition is trivial for other objects), we identify the relative slice as
\begin{equation*}
\Ar(\I)\times_{(\Ar(\I^{\triangleleft})\times_{\I^{\triangleleft}}\I)}(\Ar(\I^{\triangleleft})\times_{\I^{\triangleleft}}\I)_{(-\infty\rightarrow i)/}\simeq\Ar(\I)\times_{\I}\I_{i/},
\end{equation*}
where the pullback is taken over the cocartesian fibration $\ev_1\colon\Ar(\I)\rightarrow\I$. For any map $\varphi\colon i\rightarrow j$ in $\I$, the associated cocartesian transport is the post-composition map $\varphi_{\circ}\colon\I_{/i}\rightarrow\I_{/j}$, which is coinitial as $\I$ is cofiltered, so $\ev_1$ is proper by Proposition \ref{CocartesianProper}. In light of the diagram
\begin{equation*}
\begin{tikzcd}
\I_{/i}\arrow[r]\arrow[d]\arrow[dr,very near start,phantom,"\lrcorner"] & \Ar(\I)\times_{\I}\I_{i/}\arrow[r]\arrow[d]\arrow[dr,very near start,phantom,"\lrcorner"] & \Ar(\I)\arrow[d,"\ev_1"]\\
\{\id_i\}\arrow[r] & \I_{i/}\arrow[r] & \I,
\end{tikzcd}
\end{equation*}
in which the squares are cartesian and the bottom left horizontal map is coinitial, this implies that the upper left horizontal map is coinitial. Returning to the existence criterion for $p$-right Kan extensions, we have to establish the existence of the dotted map which is a $p$-limit diagram in
\begin{equation*}
\begin{tikzcd}
\I_{/i}\arrow[r]\arrow[d] & \Ar(\I)\times_{\I}\I_{i/}\arrow[r] & \Un(\C_{\bullet})\arrow[d,"p"]\\
\I_{/i}^{\triangleleft}\arrow[rr,bend right=1.5em,"\const_i"']\arrow[r]\arrow[urr,dashed] & \br{\Ar(\I)\times_{\I}\I_{i/}}^{\triangleleft}\arrow[r]\arrow[ur,dotted]\arrow[from=u,crossing over] & \I.
\end{tikzcd}
\end{equation*}
The top left horizontal map is coinitial, so it suffices to establishes the existence of the dashed map which is a $p$-limit diagram by \cite[\href{https://kerodon.net/tag/02XQ}{Tag 02XQ}]{kerodon}. Since the bottom horizontal map is constant at $i$, this reduces to the existence of $\I_{/i}$-shaped limits in the fiber $\C_i$ by (the dual of) \cite[\href{https://kerodon.net/tag/069Z}{Tag 069Z}]{kerodon}. The category $\I$ is cofiltered, hence the projection $\I_{/i}\rightarrow\I$ is coinitial and this existence is the assumption that $\C_i\in\widehat{\Cat}\sqb{\I}$.

There is a relative adjunction
\begin{equation*}
\begin{tikzcd}
\I^{\triangleleft}\arrow[rr,shift left=0.2em,"c"]\arrow[dr,symbol={=}] & & \Ar(\I^{\triangleleft})\arrow[ll,shift left=0.2em,"\ev_1"]\arrow[dl,"\ev_1"]\\
& \I^{\triangleleft} & 
\end{tikzcd}
\end{equation*}
given by fiberwise selecting the cone point. The composite $\overline{t}c\vert_{\I}\colon\I\rightarrow\Ar(\I^{\triangleleft})\times_{\I^{\triangleleft}}\I\rightarrow\Un(\C_{\bullet})$ is our candidate for the reflection $R(t)$. Thus we have to verify that $\overline{t}c\vert_{\I}$ is a cocartesian section and satisfies the correct universal property.

To verify that the section $\overline{t}c\vert_{\I}$ is cocartesian, we unwind the previous construction: on objects, it is given by
\begin{equation*}
\overline{t}c(i)\simeq\mathop{p\mhyphen\lim}\limits_{(\varphi\colon j\rightarrow k,\psi\colon i\rightarrow k)\in\Ar(\I)\times_{\I}\I_{i/}}\varphi_{\#}t(j)\simeq\lim_{(\varphi\colon j\rightarrow i)\in\I_{/i}}\varphi_{\#}t(j),\qquad i\in\I,
\end{equation*}
and a map $\psi\colon i\rightarrow i^{\prime}$ in $\I$ induces the diagram (by making precise the naturality of \cite[\href{https://kerodon.net/tag/069Z}{Tag 069Z}]{kerodon} in the fiber)
\begin{equation*}
\begin{tikzcd}[column sep=small]
\mathop{p\mhyphen\lim}\limits_{(\varphi\colon j\rightarrow k, \psi\colon i\rightarrow k)\in\Ar(\I)\times_{\I}\I_{i/}}\varphi_{\#}t(j)\arrow[rr,"(\Ar(\I)\times_{\I}\psi^{\circ})^{\ast}"]\arrow[dd,"\sim"{sloped}] & & \mathop{p\mhyphen\lim}\limits_{(\varphi^{\prime}\colon j^{\prime}\rightarrow k^{\prime}, \psi^{\prime}\colon i^{\prime}\rightarrow k^{\prime})\in\Ar(\I)\times_{\I}\I_{i^{\prime}/}}\varphi^{\prime}_{\#}t(j^{\prime})\arrow[d,"\sim"{sloped}]\\
& & \lim\limits_{(\varphi^{\prime}\colon j^{\prime}\rightarrow i^{\prime})\in\I_{/i^{\prime}}}\varphi^{\prime}_{\#}t(j^{\prime})\arrow[d,"\sim"{sloped}]\\
\lim\limits_{(\varphi\colon j\rightarrow i)\in\I_{/i}}\varphi_{\#}t(j)\arrow[r,"\mathrm{cocart}"] & \psi_{\#}\br{\lim\limits_{(\varphi\colon j\rightarrow i)\in\I_{/i}}\varphi_{\#}t(j)}\arrow[r,"\sim"] & \lim\limits_{(\varphi\colon j\rightarrow i)\in\I_{/i}}\psi_{\#}\varphi_{\#}t(j),
\end{tikzcd}
\end{equation*}
where the maps to the lower right corner are given by restriction along $\psi_{\circ}\colon\I_{/i}\rightarrow\I_{/j}$ resp. the limit-comparison map for $\psi_{\#}$. These are equivalences since $\psi_{\circ}$ is coinitial (as $\I$ is cofiltered) resp. since $\psi_{\#}$ is a morphism in $\widehat{\Cat}\sqb{\I}$ and $\I_{/i}\rightarrow\I$ is coinitial (as $\I$ is cofiltered). Thus, $\overline{t}c\vert_{\I}\colon\I\rightarrow\Un(\C_{\bullet})$ is a cocartesian section.

To verify the universal property, we calculate, naturally in a cocartesian section $s\colon\I\rightarrow\Un(\C_{\bullet})$ of $p$, that
\begin{align}
\Map_{\Fun_{/\I}^{\cc}(\I,\Un(\C_{\bullet}))}(s,\overline{t}c\vert_{\I})&\simeq\Map_{\Fun_{/\I}(\I,\Un(\C_{\bullet}))}(s,\overline{t}c\vert_{\I})\nonumber\\
&\simeq\Map_{\Fun_{/\I}(\Ar(\I^{\triangleleft})\times_{\I^{\triangleleft}}\I,\Un(\C_{\bullet}))}(s\ev_1\vert_{\I},\overline{t})\\
&\simeq\Map_{\Fun_{/\I}(\Ar(\I),\Un(\C_{\bullet}))}(s\ev_1,\tilde{t}_1)\\
&\simeq\Map_{\Fun(\Ar(\I),\Un(\C_{\bullet}))}(s\ev_1,\tilde{t}_1)\times_{\Map_{\Fun(\Ar(\I),\I)}(\ev_1,\ev_1)}\{\ast\}\nonumber\\
&\simeq\Map_{\Fun(\Ar(\I),\Un(\C_{\bullet}))}(s\ev_0,\tilde{t}_1)\times_{\Map_{\Fun(\Ar(\I),\I)}(\ev_0,ev_1)}\{\ast\}\\
&\simeq\Map_{\Fun(\I,\Un(\C_{\bullet}))}(s,\tilde{t}_1\Delta)\times_{\Map_{\Fun(\I,\I)}(\id_{\I},\id_{\I})}\{\ast\}\\
&\simeq\Map_{\Fun_{/\I}(\I,\Un(\C_{\bullet}))}(s,t).\nonumber
\end{align}
These natural equivalences are as follows:
\begin{enumerate}[label=(\arabic*)]
\item Use the image of the adjunction $c\vert_{\I}\dashv\ev_1\vert_{\I}$ under the contravariant $2$-functor $\Fun_{/\I}(-,\Un(\C_{\bullet}))$.
\item Use the universal property of the $p$-right Kan extension.
\item Use that the canonical transformation $s\ev_0\rightarrow s\ev_1$ is a cocartesian morphism for the cocartesian fibration $p_{\circ}\colon\Fun(\Ar(\I),\Un(\C_{\bullet}))\rightarrow\Fun(\Ar(\I),\I)$ since $s$ is a cocartesian section of $p$.
\item Use the image of the adjunction $\Delta\dashv\ev_0$ under the contravariant $2$-functor $\Fun(-,\Un(\C_{\bullet}))$.
\end{enumerate}
This concludes the proof, and naturality is clear by construction.
\end{proof}
\begin{remark}
The proof uses that $\I$ is cofiltered in two ways, namely that the post-composition maps $\psi_{\circ}\colon\I_{/i}\rightarrow\I_{/j}$ for $\psi\colon i\rightarrow j$ in $\I$ are coinitial, and that the projections $\I_{/i}\rightarrow\I$ for $i\in\I$ are coinitial. The latter condition was only used to establish the existence or preservation of limits, so we can drop this condition if we assume the diagram $\C_{\bullet}$ to be $\widehat{\Cat}\sqb{\K}$-valued for $\K=\{\I_{/i}\colon i\in\I\}$ instead. The former condition, however, is essential. It characterizes precisely the class of categories that are \textit{codistilled} in the sense of \cite[Section 4]{Rezk} or \textit{coconfluent} in the sense of \cite{SattlerWärn} (cf. \cite[Section 9.8]{HaugsengYetAnother}). This class includes cofiltered categories, anima and in fact consists exactly of the anima-indexed limits of cofiltered categories \cite[Corollary 16.9]{Rezk}. The results in this paper all admit generalizations to such indexing categories, which we do not state explicitly for the sake of simplicity.
\end{remark}

Finally, we state the dual\footnote{In the sense that these correspond under applying the inverse equivalences $(-)^{\op}\colon\widehat{\Cat}\br{\I}\stackrel{\sim}{\longleftrightarrow}\widehat{\Cat}\sqb{\I}\colon(-)^{\op}$ to the diagrams.} of Theorem \ref{MainThm}.
\begin{theorem}\label{MainThmDual}
Let $\I$ be a small, filtered category and $\C_{\bullet}\colon\I^{\op}\rightarrow\widehat{\Cat}\br{\I}$ a diagram. Then $\lim_{\I^{\op}}\C_{\bullet}\subseteq\lim_{\I^{\op}}^{\oplax}\C_{\bullet}$ admits a left adjoint $L\colon\lim^{\oplax}_{\I^{\op}}\C_{\bullet}\rightarrow\lim_{\I^{\op}}\C_{\bullet}$, which is given levelwise by
\begin{equation*}
\pi_iL\simeq\colim_{(\varphi\colon i\rightarrow j)\in\I_{i/}}\varphi^{\ast}\pi_j\colon\lim\nolimits_{\I^{\op}}^{\oplax}\C_{\bullet}\rightarrow\C_i\qquad\text{ for all }i\in\I.
\end{equation*}
Furthermore, the resulting adjunction is natural in the diagram $\C_{\bullet}$.\qed
\end{theorem}
\begin{addendum}\label{LocalizationCoUnit}
Furthermore, unwinding the proof yields explicit descriptions of the (co-)unit maps:
\begin{itemize}[leftmargin=*]
\item The counit map $L\vert_{\lim_{\I^{\op}}\C_{\bullet}}\rightarrow\id_{\lim_{\I^{\op}}\C_{\bullet}}$ is given levelwise by the equivalences
\begin{equation*}
\pi_iL\vert_{\lim_{\I^{\op}}\C_{\bullet}}\simeq\colim_{(\varphi\colon i\rightarrow j)\in\I_{i/}}\varphi^{\ast}\pi_j\vert_{\lim_{\I^{\op}}\C_{\bullet}}\simeq\colim_{(\varphi\colon i\rightarrow j)\in\I_{i/}}\pi_i\vert_{\lim_{\I^{\op}}\C_{\bullet}}\simeq\pi_i\vert_{\lim_{\I^{\op}}\C_{\bullet}}\qquad\text{ for all }i\in\I,
\end{equation*}
using that $\pi_i\vert_{\lim_{\I^{\op}}\C_{\bullet}}\iso\varphi^{\ast}\pi_j\vert_{\lim_{\I^{\op}}\C_{\bullet}}$ by the definition of cartesian sections, and that $\I_{i/}$ is weakly contractible.
\item The unit map $\id_{\lim_{\I^{\op}}^{\oplax}\C_{\bullet}}\rightarrow L$ is given levelwise by the maps
\begin{equation*}
\pi_i\rightarrow\colim_{(\varphi\colon i\rightarrow j)\in\I_{i/}}\varphi^{\ast}\pi_j\simeq\pi_iL
\end{equation*}
corresponding to the $\id_i$-th component of the colimit cocone, for all $i\in\I$.
\end{itemize}
\end{addendum}

\begin{addendum}\label{NaturalityInShape}
Following Addendum \ref{LocalizationCoUnit}, we can also obtain naturality in the diagram shape. Let
\begin{equation*}
\begin{tikzcd}
\I^{\op}\arrow[r,"\C_{\bullet}"]\arrow[d,"F"] & \widehat{\Cat}(\I)\arrow[d,hookrightarrow]\\
\I^{\prime,\op}\arrow[r,"\C_{\bullet}^{\prime}"] & \widehat{\Cat}(\I^{\prime})
\end{tikzcd}
\end{equation*}
be a diagram, so we obtain a commutative square
\begin{equation*}
\begin{tikzcd}[column sep=large]
\lim_{\I^{\prime,\op}}\C_{\bullet}^{\prime}\arrow[r,hookrightarrow]\arrow[d,"F^{\ast}"] & \lim_{\I^{\prime,\op}}^{\oplax}\C_{\bullet}^{\prime}\arrow[l,bend right=2em,dotted,"L^{\prime}"']\arrow[d,"F^{\ast}"]\\
\lim_{\I^{\op}}\C_{\bullet}\arrow[r,hookrightarrow] & \lim_{\I^{\op}}^{\oplax}\C_{\bullet}.\arrow[l,bend right=2em,dotted,"L" above]
\end{tikzcd}
\end{equation*}
Then the induced mate transformation $LF^{\ast}\Rightarrow F^{\ast}L^{\prime}$ is given levelwise by the canonical maps
\begin{equation*}
\begin{tikzcd}
\pi_iLF^{\ast}\simeq\colim\limits_{(\varphi\colon i\rightarrow j)\in\I_{i/}}F(\varphi)^{\ast}\pi_jF^{\ast}\arrow[r] & \pi_IF^{\ast}L^{\prime}\simeq\colim\limits_{(\varphi^{\prime}\colon F(i)\rightarrow j^{\prime})\in\I_{F(i)/}^{\prime}}\varphi^{\prime,\ast}\pi_{j^{\prime}}^{\prime}
\end{tikzcd}
\end{equation*}
induced on the indexing categories by $F_{i/}\colon\I_{i/}\rightarrow\I^{\prime}_{F(i)/}$, for all $i\in\I$.
\end{addendum}

\begin{example}
The application of Theorem \ref{MainThmDual} to the diagram
\begin{equation*}
\begin{tikzcd}
\dotsc\arrow[r,"\Omega"] & \An_{\ast}\arrow[r,"\Omega"] & \An_{\ast}
\end{tikzcd}
\end{equation*}
recovers the classical formula for the spectrification of a pre-spectrum, see e.g. \cite[p. 3]{LMS}. This example also already appears in \cite[Section 35.5]{Joyal}.
\end{example}

%\section{Adjunctions}

Now that Theorem \ref{IntroReflection} has been obtained, we derive Theorem \ref{IntroAdjunction} in the indicated manner. First, we observe that adjunctions are always compatible with \textit{lax} limits (see e.g. \cite[Proposition 5.1]{HorevYanovski}, \cite[Lemma 3.19]{Linskens}).

\begin{proposition}\label{AdjunctionOnLaxLimits}
Let $\I$ be a small category, $\C_{\bullet},\D_{\bullet}\colon\I\rightarrow\widehat{\Cat}$ be two diagrams of categories and $L\colon\C_{\bullet}\Rightarrow\D_{\bullet}$ a natural transformation that is pointwise given by left adjoints. Then $\lim^{\lax}_{\I}L\colon\lim^{\lax}_{\I}\C_{\bullet}\rightarrow\lim^{\lax}_{\I}\D_{\bullet}$ is again a left adjoint.

More precisely, the right adjoints $R_i\colon\D_i\rightarrow\C_i,\,i\in\I$ assemble into a lax natural transformation $R\colon\D_{\bullet}\Rightarrow^{\lax}\C_{\bullet}$, which induces the right adjoint $\lim^{\lax}_{\I}R\colon\lim^{\lax}_{\I}\D_{\bullet}\rightarrow\lim^{\lax}_{\I}\C_{\bullet}$.

Furthermore, a natural transformation of such natural transformations that satisfies the Beck-Chevalley condition levelwise induces on lax limits a commutative square satisfying the Beck-Chevalley condition.
\end{proposition}
\begin{proof}[Proof.]
The dual of \cite[Proposition 7.3.2.6]{HA} implies that there is a relative adjunction
\begin{equation*}
\begin{tikzcd}
\Un(\C_{\bullet})\arrow[rr,shift left=0.2em,"L"]\arrow[dr,"p"'] & & \Un(\D_{\bullet})\arrow[ll,shift left=0.2em,"R"]\arrow[dl,"q"]\\
& \I & 
\end{tikzcd}
\end{equation*}
in which $R$ is given fiberwise by the $R_i,\,i\in\I$ and describes the lax natural transformation $R\colon\D_{\bullet}\Rightarrow^{\lax}\C_{\bullet}$. Then applying the $2$-functor $\Fun_{/\I}(\I,-)$ yields the desired adjunction. The additional naturality statement is a consequence of the fact that the jointly conservative $2$-functors $\Fun_{/\I}(\{i\},-),\,i\in\I$ preserve the mate transformation.
\end{proof}

\begin{remark}\label{AdjunctionOnLaxLimitsRmk}
Additionally, in this situation, the adjunction $\begin{tikzcd}\lim^{\lax}_{\I}L\colon\lim^{\lax}_{\I}\C_{\bullet}\Adj & \lim^{\lax}_{\I}\D_{\bullet}\colon\lim^{\lax}_{\I}R\end{tikzcd}$ restricts to an adjunction $\begin{tikzcd}\lim_{\I}L\colon\lim_{\I}\C_{\bullet}\Adj & \lim_{\I}\D_{\bullet}\end{tikzcd}$ if and only if the naturality squares
\begin{equation*}
\begin{tikzcd}
\C_i\arrow[r,"L_i"]\arrow[d,"\varphi_{\#}"] & \D_i\arrow[d,"\varphi_{\#}"]\\
\C_j\arrow[r,"L_j"] & \D_j
\end{tikzcd}
\end{equation*}
are horizontally right adjointable for all $\varphi\colon i\rightarrow j$ in $\I$. In this case, the lax natural transformation $R\colon G\Rightarrow^{\lax}F$ is a posteriori a (strong) natural transformation and induces the right adjoint $\lim_{\I}R\colon\lim_{\I}\D_{\bullet}\rightarrow\lim_{\I}\C_{\bullet}$.

%The sufficiency of this condition also admits a direct proof. Indeed, the natural transformation can be curried to a functor $\I\rightarrow\Fun(\sqb{1},\widehat{\Cat})\simeq\widehat{\mathrm{Cart}}(\sqb{1}^{\op})$ and the assumptions are equivalent to this functor factoring through the (non-full) subcategory $\widehat{\mathrm{Bicart}}(\sqb{1}^{\op})\hookrightarrow\widehat{\mathrm{Cart}}(\sqb{1}^{\op})$, which is stable under (small) limits.
\end{remark}

Nonetheless, it is possible for $\lim_{\I}L\colon\lim_{\I}\C_{\bullet}\rightarrow\lim_{\I}\D_{\bullet}$ to admit a right adjoint not of this form, e.g. by using the colocalization of Theorems \ref{MainThm} as we are about to. A similar strategy is used in \cite{HorevYanovski} to prove their \singlequote{adjoint descent}, which makes no assumptions on $\I$ but instead assumes that the diagram $\C_{\bullet}$ is constant.

\begin{theorem}\label{Adjunction}
Let $\I$ be a small, cofiltered category, $\C_{\bullet},\D_{\bullet}\colon\I\rightarrow\widehat{\Cat}$ be two diagrams of categories such that $\C_{\bullet}$ factors through $\widehat{\Cat}\sqb{\I}\hookrightarrow\widehat{\Cat}$ and $L\colon\C_{\bullet}\Rightarrow\D_{\bullet}$ a natural transformation that is pointwise given by left adjoints. Then there is an induced adjunction
\begin{equation*}
\begin{tikzcd}
\lim_{\I}L\colon\lim_{\I}\C_{\bullet}\Adj & \lim_{\I}\D_{\bullet}\colon R.
\end{tikzcd}
\end{equation*}
More precisely, the right adjoint $R$ is given levelwise by
\begin{equation*}
\pi_iR\simeq\lim_{(\varphi\colon j\rightarrow i)\in\I_{/i}}\varphi_{\#}R_j\pi_j\colon\lim\nolimits_{\I}\D_{\bullet}\rightarrow\C_i\qquad\text{ for all }i\in\I.
\end{equation*}
Furthermore, a natural transformation of such natural transformations that satisfies the Beck-Chevalley condition levelwise induces on limits a commutative square satisfying the Beck-Chevalley condition.
\end{theorem}
\begin{remark}
In light of the formula $\lim^{\lax}_{\I}\C_{\bullet}\simeq\lim_{\I}\Fun_{/\I}(\I_{\bullet/},\Un(\C_{\bullet}))$ under which $\lim_{\I}\C_{\bullet}\subseteq\lim^{\lax}_{\I}\C_{\bullet}$ is induced by the fully faithful left adjoints $\C_i\hookrightarrow\Fun_{/\I}(\I_{i/},\Un(\C_{\bullet}))$ to evaluation at $\id_i$, we could also deduce Theorem \ref{MainThm} from Theorem \ref{Adjunction}, i.e. the two theorems are logically equivalent.
\end{remark}
\begin{proof}[Proof.]
Take the composite of the adjunctions
\begin{equation*}
\begin{tikzcd}
\lim_{\I}\C_{\bullet}\arrow[r,shift left=0.2em,hookrightarrow] & \lim^{\lax}_{\I}\C_{\bullet}\arrow[l,shift left=0.2em]\arrow[r,"\lim^{\lax}_{\I}L",shift left=0.2em] & \lim^{\lax}_{\I}\D_{\bullet}\arrow[l,"\lim^{\lax}_{\I}R",shift left=0.2em]
\end{tikzcd}
\end{equation*}
of Theorem \ref{MainThm} and Proposition \ref{AdjunctionOnLaxLimits} respectively. The left adjoint factors through $\lim_{\I}\D_{\bullet}\subseteq\lim^{\lax}_{\I}\D_{\bullet}$, so the adjunction restricts to the desired one. The naturality is also immediate from these results.
\end{proof}

Finally, we state the dual of Theorem \ref{Adjunction}.
\begin{theorem}\label{AdjunctionDual}
Let $\I$ be a small, filtered category, $\C_{\bullet},\D_{\bullet}\colon\I^{\op}\rightarrow\widehat{\Cat}$ be two diagrams of categories such that $\D_{\bullet}$ factors through $\widehat{\Cat}\br{\I}\hookrightarrow\widehat{\Cat}$ and $R\colon\D_{\bullet}\Rightarrow\C_{\bullet}$ a natural transformation that is pointwise given by right adjoints. Then there is an induced adjunction
\begin{equation*}
\begin{tikzcd}
L\colon\lim_{\I^{\op}}\C_{\bullet}\Adj & \lim_{\I^{\op}}\D_{\bullet}\colon\lim_{\I^{\op}}R.
\end{tikzcd}
\end{equation*}
More precisely, the left adjoint $L$ is given levelwise by
\begin{equation*}
\pi_iL\simeq\colim_{(\varphi\colon i\rightarrow j)\in\I_{i/}}\varphi^{\ast}L_j\pi_j\colon\lim\nolimits_{\I^{\op}}\C\rightarrow\D_i\qquad\text{ for all }i\in\I.
\end{equation*}
Furthermore, a natural transformation of such natural transformations that satisfies the Beck-Chevalley condition levelwise induces on limits a commutative square satisfying the Beck-Chevalley condition.\qed
\end{theorem}
\begin{addendum}\label{AdjunctionCoUnit}
Following Addendum \ref{LocalizationCoUnit}, we also obtain explicit descriptions of the (co-)unit maps:
\begin{itemize}[leftmargin=*]
\item The counit map $L(\lim_{\I^{\op}}R)\rightarrow\id_{\lim_{\I^{\op}}\D_{\bullet}}$ is given levelwise by the maps (we abbreviate $\lim_{\I^{\op}}R$ to $R$)
\begin{equation*}
\pi_iLR\simeq\colim_{(\varphi\colon i\rightarrow j)\in\I_{i/}}\varphi^{\ast}L_j\pi_jR\simeq\colim_{(\varphi\colon i\rightarrow j)\in\I_{i/}}\varphi^{\ast}L_jR_j\pi_j\rightarrow\colim_{(\varphi\colon i\rightarrow j)\in\I_{i/}}\varphi^{\ast}\pi_j\simeq\colim_{(\varphi\colon i\rightarrow j)\in\I_{i/}}\pi_i\simeq\pi_i,
\end{equation*}
where the non-invertible transformation is induced by the counits of the pointwise adjunctions, for all $i\in\I$.
\item The unit map $\id_{\lim_{\I^{\op}}\C_{\bullet}}\rightarrow(\lim_{\I^{\op}}R)L$ is given levelwise by the maps (we abbreviate $\lim_{\I^{\op}}R$ to $R$)
\begin{equation*}
\pi_i\rightarrow R_iL_i\pi\rightarrow\colim_{(\varphi\colon i\rightarrow j)\in\I_{i/}}\varphi^{\ast}R_jL_j\pi_j\simeq\colim_{(\varphi\colon i\rightarrow j)\in\I_{i/}}R_i\varphi^{\ast}L_j\pi_j\simeq R_i\colim_{(\varphi\colon i\rightarrow j)\in\I_{i/}}\varphi^{\ast}L_j\pi_j\simeq R_i\pi_iL\simeq\pi_iRL,
\end{equation*}
where the first transformation is induced by the units of the pointwise adjunctions and the second transformation corresponds to the $\id_i$-th component of the colimit cocone, for all $i\in\I$.
\end{itemize}
\end{addendum}

\begin{addendum}\label{NaturalityOfAdjunctionInShape}
Following Addendum \ref{NaturalityInShape}, we also obtain naturality in the diagram shape. Let
\begin{equation*}
\begin{tikzcd}
\Bigg(\I^{\op}\arrow[r,"F"] & \I^{\prime,\op}\arrow[r,"\C_{\bullet}^{\prime}",""{below,name=U},bend left=2.5em]\arrow[r,"\D_{\bullet}^{\prime}"',""{above,name=D},bend right=2.5em]\arrow[from=U,to=D,Rightarrow,"R^{\prime}"] & \widehat{\Cat}\Bigg)\arrow[r,symbol=\simeq] & \Bigg(\I^{\op}\arrow[r,"\C_{\bullet}",""{below,name=U},bend left=2.5em]\arrow[r,"\D_{\bullet}"',""{above,name=D},bend right=2.5em]\arrow[from=U,to=D,Rightarrow,"R"] & \widehat{\Cat}\Bigg)
\end{tikzcd}
\end{equation*}
be a diagram such that $\C_{\bullet}^{\prime}$ factors through $\widehat{\Cat}(\I^{\prime})\hookrightarrow\widehat{\Cat}$ and $\C_{\bullet}$ factors through $\widehat{\Cat}(\I)\hookrightarrow\widehat{\Cat}$, so we obtain a commutative square
\begin{equation*}
\begin{tikzcd}[column sep=large]
\lim_{\I^{\prime,\op}}\D_{\bullet}^{\prime}\arrow[r,"\lim_{\I^{\prime,\op}}R^{\prime}"']\arrow[d,"F^{\ast}"] & \lim_{\I^{\prime,\op}}\C_{\bullet}^{\prime}\arrow[d,"F^{\ast}"]\arrow[l,"L^{\prime}"',bend right=2em,dotted]\\
\lim_{\I^{\op}}\D_{\bullet}\arrow[r,"\lim_{\I^{\op}}R"'] & \lim_{\I^{\op}}\C_{\bullet}.\arrow[l,"L"',bend right=2em,dotted]
\end{tikzcd}
\end{equation*}
Then the induced mate transformation $LF^{\ast}\Rightarrow F^{\ast}L^{\prime}$ is given levelwise by the maps
\begin{equation*}
\begin{tikzcd}
\pi_iLF^{\ast}\simeq\colim\limits_{(\varphi\colon i\rightarrow j)\in\I_{i/}}F(\varphi)^{\ast}L_{F(j)}\pi_jF^{\ast}\arrow[r] & \pi_iF^{\ast}L^{\prime}\simeq\colim\limits_{(\varphi^{\prime}\colon F(i)\rightarrow j^{\prime})\in\I^{\prime}_{F(i)/}}\varphi^{\prime,\ast}L_{j^{\prime}}\pi_{j^{\prime}}^{\prime}
\end{tikzcd}
\end{equation*}
induced on the indexing categories by $F_{i/}\colon\I_{i/}\rightarrow\I^{\prime}_{F(i)/}$, for all $i\in\I$.
\end{addendum}

\section{Applications}\label{Applications}

In this section, we give a number of applications of our main results. First, we illustrate a number of \singlequote{stability} results. Some of these only use the existence of adjunctions (which in many situations may also be accessible by other methods, e.g. the adjoint functor theorem), but others crucially rely on the explicit formulas we have obtained. In the latter part of the section, we then prove the push-pull formula (Proposition \ref{IntroPushPull}) and explore its consequences.

The next proposition invokes the theory of lax-idempotent monads in $2$-categories (forthcoming \cite{AbellanBlom}, cf. \cite{BlomSlides}).
\begin{proposition}\label{LaxIdempotentMonads}
Let $\I$ be a small, filtered category, $T$ a lax-idempotent monad on the $2$-category $\widehat{\mathbf{Cat}}\br{\I}$ and $\C_{\bullet}\colon\I^{\op}\rightarrow\widehat{\Cat}\br{\I}$ a diagram of categories. If $\C_i$ is a $T$-algebra for all $i\in\I$, then $\lim_{\I^{\op}}\C_{\bullet}$ is a $T$-algebra.

Furthermore, a natural transformation of such diagrams that is given by $T$-algebra homomorphisms levelwise induces on limits a $T$-algebra homomorphism.
\end{proposition}
\begin{proof}[Proof.]
This is a consequence of the fact that a category $\C\in\widehat{\mathbf{Cat}}(\I)$ is a $T$-algebra (this is indeed a property) if and only if the unit map $\eta\colon\C\rightarrow T\C$ admits a left adjoint and a functor $F\colon\C\rightarrow\D$ of $T$-algebras is a $T$-algebra homomorphism if and only if the resulting square $(F,TF)\colon(\eta_{\C}\colon\C\rightarrow T\C)\rightarrow(\eta_{\D}\colon\D\rightarrow T\D)$ satisfies the Beck-Chevalley condition, by applying Theorem \ref{AdjunctionDual} to the natural transformation $\eta\colon\C_{\bullet}\rightarrow T\C_{\bullet}$.
\end{proof}

\begin{corollary}\label{ApplicationCocomplete}
Let $\I$ be a small, filtered category, $\C_{\bullet}\colon\I^{op}\rightarrow\widehat{\Cat}\br{\I}$ a diagram and $\K$ a collection of (small) categories. If $\C_i$ admits $\K$-shaped colimits for all $i\in\I$, then $\lim_{\I^{\op}}\C_{\bullet}$ admits $\K$-shaped colimits.

Furthermore, a natural transformation of such diagrams that is given by $\K$-cocontinuous functors levelwise induces on limits a $\K$-cocontinuous functor.
\end{corollary}
\begin{proof}[Proof.]
This follows by applying Proposition \ref{LaxIdempotentMonads} to the lax-idempotent monad that freely adds $\K\cup\{\I\}$-shaped colimits whilst preserving $\I$-shaped colimits. Alternatively, we can directly apply Theorem \ref{AdjunctionDual} to the natural transformations $\const\colon\C_{\bullet}\Rightarrow\Fun(K,\C_{\bullet})$ for all $K\in\K$.
\end{proof}

%\begin{example}
%Applying this to the diagram
%\begin{equation*}
%\begin{tikzcd}
%\dotsc\arrow[r,"\Omega"] & \An_{\ast}\arrow[r,"\Omega"] & \An_{\ast},
%\end{tikzcd}
%\end{equation*}
%we obtain that the category $\Sp$ of spectra is cocomplete (and a way of obtaining colimits therein), which is not entirely obvious.
%\end{example}

\begin{proposition}\label{CartesianClosed}
Let $\I$ be a small, cofiltered category and $\C_{\bullet}\colon\I\rightarrow\widehat{\Cat}\sqb{\{\I,\ast\sqcup\ast\}}$ be a diagram. If $\C_i$ is cartesian closed for all $i\in\I$, then $\lim_{\I}\C_{\bullet}$ is cartesian closed.

Furthermore, a natural transformation of such diagrams that is given by cartesian closed functors levelwise induces on limits a cartesian closed functor.
\end{proposition}
\begin{proof}[Proof.]
The diagonal is a natural transformation $\Delta\colon\C_{\bullet}\Rightarrow\C_{\bullet}^{\times2}$ whose naturality squares satisfy the Beck-Chevalley condition as the diagram factors through $\widehat{\Cat}\sqb{\ast\sqcup\ast}\hookrightarrow\widehat{\Cat}$, so there is an adjoint transformation $(-)\times(-)\colon \C_{\bullet}^{\times2}\Rightarrow \C_{\bullet}$. For any $x\in\lim_{\I}\C_{\bullet}$ corresponding to a natural transformation $x\colon\ast\Rightarrow \C_{\bullet}$
, we consider the composite transformation
\begin{equation*}
\begin{tikzcd}[column sep=large]
\C_{\bullet}\arrow[r,Rightarrow,"x\times\id"] & \C_{\bullet}^{\times2}\arrow[r,Rightarrow,"(-)\times(-)"] & \C_{\bullet}.
\end{tikzcd}
\end{equation*}
This natural transformation is given levelwise by the maps $\pi_i(x)\times(-)\colon\C_i\rightarrow\C_i,\,i\in\I$, which admit right adjoints by assumption. Then Theorem \ref{Adjunction} implies that the induced map $x\times(-)\colon\lim_{\I}\C_{\bullet}\rightarrow\lim_{\I}\C_{\bullet}$ on limits also admits a right adjoint, i.e. $\lim_{\I}\C_{\bullet}$ is cartesian closed.
\end{proof}
%\begin{example}
%The category $\omega\Cat\simeq(\dotsc\rightarrow(n+1)\Cat\rightarrow n\Cat\rightarrow\dotsc\rightarrow1\Cat\rightarrow0\Cat)$ is cartesian closed.
%\end{example}

\begin{proposition}
Let $\kappa$ be a regular cardinal, $\I$ a small, $\kappa$-filtered category and $\C_{\bullet}\colon\I^{\op}\rightarrow\widehat{\mathrm{Acc}}_{\kappa}$ a diagram of $\kappa$-accessible categories and $\kappa$-accessible functors. If $\C_i$ is a presentable category for all $i\in\I$, then $\lim_{\I}\C_{\bullet}$ is a presentable category.
\end{proposition}
\begin{proof}[Proof.]
The oplax limit $\lim_{\I^{\op}}^{\oplax}\C_{\bullet}$ is an accessible category by \cite[Corollary 5.4.7.17]{HTT} and it is cocomplete by \cite[Proposition 5.1.2.2]{HTT} since the $\C_i,\,i\in\I$ are cocomplete, whence it is presentable. It then suffices to prove that $\lim_{\I^{\op}}\C_{\bullet}$ is an accessible localization of $\lim^{\oplax}_{\I^{\op}}\C_{\bullet}$, i.e. by Theorem \ref{MainThmDual} that it is closed under $\kappa$-filtered colimits. The colimits in the oplax limit are computed levelwise and the cartesian transport functors $\varphi^{\ast}\colon\C_j\rightarrow\C_i$ are $\kappa$-accessible for all maps $\varphi\colon i\rightarrow j$ in $\I$, so we see that the cartesian sections are stable under $\kappa$-filtered colimits, as desired.
\end{proof}

Furthermore, we can also use the explicit formulas to obtain non-formal properties.
\begin{proposition}\label{FullyFaithful}
In the setting of Theorem \ref{AdjunctionDual}, if the left adjoints $L_i\colon\C_i\rightarrow\D_i$ are fully faithful for all $i\in\I$, then the left adjoint $L\colon\lim_{\I^{\op}}\C_{\bullet}\rightarrow\lim_{\I^{\op}}\D_{\bullet}$ is again fully faithful.
\end{proposition}
\begin{proof}[Proof.]
It suffices to check that the unit map $\eta\colon\id_{\lim_{\I^{\op}}\C_{\bullet}}\Rightarrow(\lim_{\I^{\op}}R)L$ is an equivalence, which can be checked levelwise. To see this, consider, for any $i\in\I$, the commutative square
\begin{equation*}
\begin{tikzcd}
\pi_i\arrow[r]\arrow[d,"\sim"{sloped}] & R_iL_i\pi_i\arrow[d]\\
\colim_{(\varphi\colon i\rightarrow j)\in\I_{i/}}\varphi^{\ast}\pi_j\arrow[r,"\sim"] & \colim_{(\varphi\colon i\rightarrow j)\in\I_{i/}}\varphi^{\ast}R_jL_j\pi_j.
\end{tikzcd}
\end{equation*}
The left vertical map is an equivalence since $\pi_i\simeq\varphi^{\ast}\pi_j$ by the definition of cartesian sections and since $\I_{i/}$ is weakly contractible. The bottom horizontal map is an equivalence since the levewise unit maps $\eta_i\colon\id_{\C_i}\rightarrow R_iL_i$ are equivalences by assumption. The composite in this square identifies with $\pi_i\eta$ by Addendum \ref{AdjunctionCoUnit}, so we conclude. 
\end{proof}

\begin{proposition}\label{FiniteLimits}
In the setting of Theorem \ref{AdjunctionDual}, assume additionally that
\begin{itemize}
\item $\K$ is a collection of finite categories and the diagrams $\C_{\bullet},\D_{\bullet}$ factor through $\widehat{\Cat}\sqb{\K}\hookrightarrow\widehat{\Cat}$,
\item the left adjoints $L_i\colon\C_i\rightarrow\D_i$ preserve $\K$-shaped limits for all $i\in\I$ and
\item $\I$-shaped colimits and $\K$-shaped limits commute in $\D_i$ for all $i\in\I$.\footnote{This is e.g. satisfied for all choices of $\I$ and $\K$ if the $\D_i,\,i\in\I$ are compactly assembled, stable or topoi.}
\end{itemize}
Then the left adjoint $L\colon\lim_{\I^{\op}}\C_{\bullet}\rightarrow\lim_{\I^{\op}}\D_{\bullet}$ preserves $\K$-shaped limits.
\end{proposition}
\begin{proof}[Proof.]
The assumption that the diagrams factor through $\widehat{\Cat}\sqb{\K}\hookrightarrow\widehat{\Cat}$ implies that $\K$-shaped limits in $\lim_{\I}\C_{\bullet}$ resp. $\lim_{\I}\D_{\bullet}$ are computed levelwise. It thus suffices to show that the composite $\pi_iL\colon\lim_{\I}\C_{\bullet}\rightarrow\D_i$ preserves $\K$-shaped limits for all $i\in\I$, and Theorem \ref{AdjunctionDual} identifies $\pi_iL\simeq\colim_{(\varphi\colon i\rightarrow j)\in\I_{i/}}\varphi_{\#}L_j\pi_j$. The functors $\varphi_{\#}$, $L_j$ and $\pi_j$ all preserve $\K$-shaped limits by assumption and $\I_{i/}\rightarrow\I$ is cofinal, so that $\I_{i/}$-shaped colimits commute with $\K$-shaped limits in $\D_i$. Together, these facts imply that $\pi_iL$ preserves $\K$-shaped limits.
\end{proof}
\begin{example}[\protect{\cite[Corollary 3.4.16]{Geiss}}]
The geometric realization functor $\abs{-}\colon\Cat\Sp\rightarrow\Sp$ from the category of categorical spectra to the category of spectra preserves finite products.
\end{example}

Finally, we turn to prove the push-pull formula mentioned in the Introduction (Proposition \ref{IntroPushPull}). More generally than explained there, we give, under the usual assumptions, a formula for left adjoints to projections of the form $\pi_i\colon\lim_{\I^{\op}}\C_{\bullet}\rightarrow\C$, where the limit is taken along right adjoints. Versions of this formula have appeared in \cite[Lemma 1.3.6]{Gaitsgory}, \cite[Proposition 1.66,2)]{Efimov} and, in the case $\I=\N$, \cite[Corollary B.0.6]{Geiss}.

\begin{proposition}[Push-Pull Formula]\label{PushPullFormula}
Let $\I$ be a small, filtered category and $\C_{\bullet}\colon\I^{\op}\rightarrow\widehat{\Cat}^R(\I)\coloneqq\widehat{\Cat}^R\cap\widehat{\Cat}\br{\I}$ a diagram. For any $i\in\I$, the projection $\pi_i\colon\lim_{\I^{\op}}\C_{\bullet}\rightarrow\C_i$ admits a left adjoint $(\pi_i)_L\colon\C_i\rightarrow\lim_{\I^{\op}}\C_{\bullet}$ and, for any $j\in\I$, the canonical comparison map
\begin{equation*}
\begin{tikzcd}
\colim\limits_{(\varphi\colon i\rightarrow k,\psi\colon j\rightarrow k)\in\I_{\{i,j\}/}}\psi^{\ast}\varphi_{\#}\arrow[r,"\sim"] & \pi_j(\pi_i)_L
\end{tikzcd}
\end{equation*}
induced by the adjunction units is a natural equivalence of functors $\C_i\rightarrow\C_j$.
\end{proposition}
\begin{proof}[Proof.]
By filteredness of $\I$, there exists a cospan $(\varphi\colon i\rightarrow k,\psi\colon j\rightarrow k)$ and the resulting map $\I_{k/}\rightarrow\I_{\{i,j\}/}$ is cofinal, so post-composing with $\psi^{\ast}$, pre-composing with $\varphi_{\#}$ and using furthermore that the projection $\I_{k/}\rightarrow\I$ is cofinal, we may assume without loss of generality that $\I$ has an initial object $\emptyset$ and $i\simeq\emptyset\simeq j$.

In this case, there is a tautological natural transformation $\C_{\bullet}\Rightarrow\const_{\C_{\emptyset}}$ of diagrams $\I^{\op}\rightarrow\widehat{\Cat}$, adjoint to the equivalence $\colim_{\I^{\op}}\C_{\bullet}\simeq\C_{\emptyset}$. The limit of this natural transformation is the projection $\pi\colon\lim_{\I}\C_{\bullet}\rightarrow\C_{\emptyset}$ and applying Theorem \ref{AdjunctionDual} yields an equivalence $\colim_{\I}(\varphi_i)^{\ast}(\varphi_i)_{\#}\simeq\pi_{\emptyset}(\pi_{\emptyset})_L$, where $\varphi_i\colon\emptyset\rightarrow i$ denotes the unique map.

It remains to identify this equivalence as induced by the canonical cocone, which we start by extending the diagram via the limit to a diagram $\overline{\C}_{\bullet}\colon(\I^{\triangleright})^{\op}\rightarrow\widehat{\Cat}^R(\I)$ and the natural transformation to $\overline{\C}_{\bullet}\Rightarrow\const_{\C_{\emptyset}}$. Then we deduce from the naturality of Addendum \ref{NaturalityOfAdjunctionInShape} for the map $\I\rightarrow\I^{\triangleright}$ that the comparison map $\colim_{\I}(\varphi_i)^{\ast}(\varphi_i)_{\#}\iso\colim_{\I^{\triangleright}}(\varphi_i)^{\ast}(\varphi_i)_{\#}$ is an equivalence, equivalently that the diagram $\I^{\triangleright}\rightarrow\Fun(\C_{\emptyset},\C_{\emptyset}),\,i\mapsto(\varphi_i)^{\ast}(\varphi_i)_{\#}$ is a colimit diagram.
\end{proof}
\begin{remark}
For completeness' sake, we explain how to unwind this diagram: it is obtained as a composite
\begin{equation*}
\begin{tikzcd}
\I^{\triangleright}\arrow[r] & \Fun\left(\lim_{(\I^{\triangleright})^{\op}}^{\oplax}\overline{\C}_{\bullet},\C_{\emptyset}\right)\arrow[r] & \Fun(\C_{\emptyset},\C_{\emptyset}),
\end{tikzcd}
\end{equation*}
where the second map is pullback along the left adjoint $\C_{\emptyset}\stackrel{\const}{\rightarrow}\Fun(\I^{\triangleright},\C_{\emptyset})\simeq\lim_{(\I^{\triangleright})^{\op}}^{\oplax}\C_{\emptyset}\rightarrow\lim_{(\I^{\triangleright})^{\op}}^{\oplax}\overline{\C}_{\bullet}$ and the first map is induced by cartesian transport, i.e. it is obtained by currying the composite
\begin{equation*}
\begin{tikzcd}
\I^{\triangleright}\times\lim_{(\I^{\triangleright})^{\op}}^{\oplax}\overline{\C}_{\bullet}\arrow[r,"\ev"] & \Un^{\mathrm{cart}}(\overline{\C}_{\bullet})\arrow[r] & \C_{\emptyset},
\end{tikzcd}
\end{equation*}
the second map being given by the the cartesian transport along the morphisms $\varphi_i\colon\emptyset\rightarrow i,\,i\in\I^{\triangleright}$. The curried composite $\I^{\triangleright}\times\C_{\emptyset}\rightarrow\I^{\triangleright}\times\lim_{(\I^{\triangleright})^{\op}}^{\oplax}\overline{\C}_{\bullet}\rightarrow\Un^{\mathrm{cart}}(\overline{\C}_{\bullet})$, on the other hand, is given by the \textit{co}cartesian transport along the morphisms $\varphi_i\colon\emptyset\rightarrow i,\,i\in\I^{\triangleright}$ by construction.
\end{remark}

\begin{remark}[Pull-Push Formula]
In this situation, we can also consider the limit cone $\const_{\lim_{\I^{\op}}\C_{\bullet}}\Rightarrow\C_{\bullet}$, whose limit is the identity on $\lim_{\I^{\op}}\C_{\bullet}$. Applying Theorem \ref{AdjunctionDual} yields an equivalence $\colim_{\I}(\pi_i)_L\pi_i\simeq\id_{\lim_{\I^{\op}}}$, which can similarly be unwound to be induced by the counit transformations. Unlike the push-pull formula, this formula is \textit{not} specific to cofiltered diagram shapes. Indeed, it can also be obtained by an application of \cite[Corollary 1.3]{HorevYanovski} in general.
\end{remark}

\begin{corollary}
Let $\I$ be a small, filtered category. The (non-full) subcategory inclusion $\widehat{\Cat}^R\hookrightarrow\widehat{\Cat}$ creates $\I^{\op}$-shaped limits along right adjoints preserving $\I$-shaped colimits. In particular, $\widehat{\Cat}^R(\I)\hookrightarrow\widehat{\Cat}$ creates all $\I^{\op}$-shaped limits.
\end{corollary}
\begin{proof}[Proof.]
For any diagram $\C_{\bullet}\colon\I^{\op}\rightarrow\widehat{\Cat}^R(\I)$, the push-pull Formula implies that the limit cone as computed in $\widehat{\Cat}$ factors through $\widehat{\Cat}^R(\I)\hookrightarrow\widehat{\Cat}$. To verify that this satisfies the universal property of a limit in $\widehat{\Cat}^R$, we need to prove that a functor $F\colon\T\rightarrow\lim_{\I^{\op}}\C_{\bullet}$ is a right adjoint if and only if the components $\pi_iF\colon\T\rightarrow\C_i$ are right adjoints for all $i\in\I$, but this is immediate by applying Theorem \ref{AdjunctionDual} to the corresponding natural transformation $\const_{\T}\Rightarrow\C_{\bullet}$.
\end{proof}

\begin{remark}\label{FilteredColimitPrL}
%The above essentially represents part of an attempt to do as much of the theory of presentable categories as possible without relying on the abstract existence results and instead providing explicit formulae. This, of course, can only go so far.
If $\C_{\bullet}\colon\I\rightarrow\Pr^L$ is a filtered diagram such that the diagram $\C_{\bullet}^R\colon\I^{\op}\rightarrow\Pr^R$ obtained by passing to right adjoints factors through $\widehat{\Cat}\br{\I}\hookrightarrow\widehat{\Cat}$ (e.g. if the original diagram is in $\Pr^L_{\ca}$ or $\Pr^{LL}_{\st}$), the push-pull formula is applicable and describes the left adjoint to the projection $\pi_i\colon\lim_{\I^{\op}}\C_{\bullet}^R\rightarrow\C_i^R$, which identifies by \cite[Corollary 5.5.3.4]{HTT} with the $i$-th component $\C_i\rightarrow\colim_{\I}^{\Pr^L}\C_{\bullet}$ of the colimit cocone of the diagram $\C_{\bullet}$ in $\Pr^L$.

In particular, in combination with Proposition \ref{FullyFaithful}, this implies that fully faithful functors in $\Pr^L_{\ca}$ or $\Pr^{LL}_{\st}$ are stable under filtered colimits (cf. \cite[Proposition 1.67]{Efimov}), taking for granted that colimits are computed underlying in $\Pr^L$.
\end{remark}

The final application we give is to \textit{shape theory}. To this end, recall that, for any topos $\X$, there exists a unique geometric morphism $p_{\ast}\colon\X\rightarrow\An$ and the corresponding monad $p_{\ast}p^{\ast}\colon\An\rightarrow\An$ is left-exact and accessible, i.e. a pro-anima. This is the \textit{shape} of $\X$, denoted $\Pi_{\infty}(\X)$, and defines a functor $\Pi_{\infty}\colon\RTop\rightarrow\Pro(\An)$.

Furthermore, a geometric morphism $f_{\ast}\colon\X\rightarrow\Y$ is called \textit{perfect} if it preserves filtered colimits \cite[Definition 24.5]{LehnerLectures}\footnote{This is related to the topological notion of perfectness by \cite[Theorem 3.50]{Lehner}. It is also implied by the stronger notion of a \textit{proper} geometric morphism by \cite[Remark 7.3.1.5]{HTT}.} and that a topos is called \textit{compact} if the terminal geometric morphism $p_{\ast}\colon\X\rightarrow\An$ is perfect (equivalently, the terminal object $1_{\X}\in\X$ is compact). Then we have the following result about the shape of cofiltered limits, adapted from \cite[Proposition 2.11, (1)]{Hoyois}. For further applications to shape theory, see \cite{Hoyois}, \cite[Lecture 24]{LehnerLectures}.

\begin{proposition}\label{ShapeCofiltered}
The shape functor $\Pi_{\infty}\colon\RTop\rightarrow\Pro(\An)$ preserves small, cofiltered limits of diagrams of compact topoi and perfect geometric morphisms.
\end{proposition}
\begin{proof}[Proof.]
Let $\I$ be a small, filtered category and $\X_{\bullet}\colon\I^{\op}\rightarrow\RTop$ a diagram of compact topoi and perfect geometric morphisms. First of all, since $\An$ is terminal in $\RTop$ and the topoi are compact, we may append the initial object $-\infty$ to $\I$ and extend the diagram via $-\infty\mapsto\An$ to obtain a diagram $(\I^{\triangleleft})^{\op}\rightarrow\widehat{\Cat}^R\cap\widehat{\Acc}_{\aleph_0}\hookrightarrow\Cat$. This computes the desired limit in $\RTop$ since $\I$ is weakly contractible and by \cite[Theorem 6.3.3.1]{HTT}.\footnote{The special case that $\RTop\hookrightarrow\widehat{\Cat}$ creates cofiltered limits along perfect geometric morphisms can also be deduced directly from our methods.}

Then the push-pull formula implies that the canonical map $\colim_{i\in\I}(p_i)_{\ast}(p_i)^{\ast}\iso p_{\ast}p^{\ast}$ is an equivalence. The subcategory $\Fun^{\mathrm{lex, acc}}(\An,\An)\subseteq\Fun(\An,\An)$ on the left-exact, accessible functors is stable under filtered colimits, so this colimit is computed in that subcategory and passing to opposites yields that the canonical map
\begin{equation*}
\begin{tikzcd}
\Pi_{\infty}(\X)\arrow[r,"\sim"] & \lim_{i\in\I^{\op}}\Pi_{\infty}(\X_i)
\end{tikzcd}
\end{equation*}
is an equivalence.
\end{proof}

\printbibliography

@book{HTT,
    author = {Lurie, Jacob},
    year = {2009},
    title = {Higher Topos Theory},
    series = {Annals of Mathematics Studies},
    number = {170},
    publisher = {Princeton University Press}
}

@misc{kerodon,
    title = {Kerodon},
    author = {Lurie, Jacob},
    howpublished = {\url{https://kerodon.net}},
    year = {2018}
}

@misc{HA,
    title = {Higher Algebra},
    author = {Lurie, Jacob},
    howpublished = {\url{https://www.math.ias.edu/~lurie/papers/HA.pdf}},
    year = {2017}
}

@mastersthesis{Geiss,
    title = {Bordism-type Theories and $(\infty,\infty)$-categories},
    author = {Geiß, Thorger},
    year = {2025},
    school = {Goethe-Universität Frankfurt}
}

@article{GHN,
    title = {Lax Colimits and Free Fibrations in $\infty$-Categories},	
    author = {Gepner, David and Haugseng, Rune and Nikolaus, Thomas},
    journal={Documenta Mathematica},
    volume={22},
    pages={1225--1266},
    year={2017}
}

@article{HHNL,
    author = {Haugseng, Rune and Hebestreit, Fabian and Linskens, Sil and Nuiten, Joost},
    title = {Lax monoidal adjunctions, two-variable fibrations and the calculus of mates},
    journal = {Proceedings of the London Mathematical Society},
    volume = {127},
    number = {4},
    pages = {889-957},
    year = {2023}
}

@article{HHNLSpans,
    title = {Two-variable fibrations, factorisation systems and $\infty $-categories of spans},
    volume = {11},
    journal = {Forum of Mathematics, Sigma},
    author = {Haugseng, Rune and Hebestreit, Fabian and Linskens, Sil and Nuiten, Joost},
    year = {2023},
    pages = {e111}
}

@book{LMS,
    title = {Equivariant Stable Homotopy Theory},
    author = {Lewis, L. Gaunce and May, J. Peter and Steinberger, Mark},
    year = {1986},
    publisher = {Springer-Verlag Berlin},
    series = {Lecture Notes in Mathematics},
    volume = {1213},
}

@misc{AbellanBlom,
    title = {Lax-idempotent monads in homotopy theory},
    author = {Abellán, Fernando and Blom, Thomas},
    year = {2026},
    note ={To appear}
}

@unpublished{BlomSlides,
    title = {Lax-idempotent monads in homotopy theory},
    author = {Blom, Thomas},
    year = {2025},
    howpublished = {\url{https://archive.math.muni.cz/conference/ct2025/data/uploads/slides/blom.pdf}}
}

@misc{HaugsengYetAnother,
    title = {Yet another introduction to $\infty$-Categories},
    author = {Haugseng, Rune},
    year = {2025},
    howpublished = {\url{https://runegha.folk.ntnu.no/naivecat_web.pdf}}
}

@unpublished{Rezk,
    title = {Generalizing Accessible $\infty$-Categories (Draft)},
    author = {Rezk, Charles},
    year = {2021},
    howpublished = {\url{https://rezk.web.illinois.edu/accessible-cat-thoughts.pdf}}
}

@misc{Efimov,
    title={K-theory and localizing invariants of large categories}, 
    author={Efimov, Alexander I.},
    year={2025},
    eprint={2405.12169},
    archivePrefix={arXiv}
}

@misc{Joyal,
    title = {Notes on Logoi},
    author = {Joyal, André},
    year = {2008},
    howpublished = {\url{https://www.math.uchicago.edu/~may/IMA/JOYAL/Joyal.pdf}}
}

@unpublished{Gaitsgory,
    title = {Notes on Geometric Langlands: Generalities on DG Categories},
    author = {Gaitsgory, Dennis},
    year = {2012},
    howpublished = {\url{https://people.mpim-bonn.mpg.de/gaitsgde/GL/textDG.pdf}}
}

@article{Hoyois,
    title = {Higher Galois theory},
    journal = {Journal of Pure and Applied Algebra},
    volume = {222},
    number = {7},
    pages = {1859-1877},
    year = {2018},
    issn = {0022-4049},
    author = {Hoyois, Marc}
}

@misc{Linskens,
    title = {Globalizing and stabilizing global $\infty$-categories}, 
    author = {Linskens, Sil},
    year = {2024},
    eprint = {2401.02264},
    archivePrefix = {arXiv}
}

@misc{SattlerWärn,
    title = {Note on Confluent Colimits},
    author = {Christian Sattler and David Wärn},
    year = {2025},
    howpublished = {\url{https://www.cse.chalmers.se/~sattler/docs/confluent/new-2025.txt}}
}

@article{HorevYanovski,
    title = {On conjugates and adjoint descent},
    journal = {Topology and its Applications},
    volume = {232},
    pages = {140-154},
    year = {2017},
    issn = {0166-8641},
    author = {Horev, Asaf and Yanovski, Lior},
}

@misc{LehnerLectures,
    author = {Lehner, Georg},
    year = {2026},
    title = {Lectures on Shape Theory},
    howpublished = {\url{https://www.dropbox.com/scl/fi/lx89delvm1z5zh1qmcabu/main.pdf?rlkey=ux5lw9hctgusfmoh2x5fahl8c&st=ocjahhrx&e=3&dl=0}}
}

@misc{Lehner,
    title = {Algebraic $K$-theory of stably compact spaces}, 
    author = {Lehner, Georg},
    year = {2026},
    eprint = {2602.18245},
    archivePrefix = {arXiv},
}
\end{document}